\documentclass[a4paper,11pt]{amsart}
\usepackage{fancyhdr}
\usepackage{amsmath}
\usepackage{dsfont}
\usepackage{hyperref}
\usepackage[mathscr]{eucal}
\usepackage[cp1251]{inputenc}
\usepackage[english]{babel}
\usepackage{cite,enumerate,float,indentfirst}
\usepackage{graphicx}
\usepackage{xcolor}
\usepackage{latexsym,a4,mathrsfs,amsthm,amsmath,amssymb,url}
\usepackage{amsfonts}
\usepackage{amssymb}
\usepackage{epstopdf}

\numberwithin{equation}{section}
\newcommand\NoBlackBoxes{\global\overfullrule0pt}
\NoBlackBoxes
\newtheorem{theorem}{Theorem}[section]
\newtheorem{lemma}[theorem]{Lemma}
\newtheorem{statement}[theorem]{Statement}

\theoremstyle{remark}
\newtheorem*{remark}{Remark}

\newcommand{\R}{\mathbb{R}}
\newcommand{\C}{\mathbb{C}}

\newcommand{\ve}{\varepsilon}

\DeclareMathOperator{\Deg}{deg}
\DeclareMathOperator{\E}{\mathbb{E}}

\DeclareMathOperator{\Pb}{\mathbb{P}}

\begin{document}

\vspace{1in}

\title[Asymptotic analysis of symmetric functions]{\bf Asymptotic analysis of symmetric functions}


\author[F. G{\"o}tze]{F. G{\"o}tze}
\address{F. G{\"o}tze\\
 Faculty of Mathematics\\
 Bielefeld University \\
 Bielefeld, 33615, Germany
}
\email{goetze@math.uni-bielefeld.de}

\author[A. Naumov]{A. Naumov}
\address{A. Naumov\\
 Faculty of Computational Mathematics and Cybernetics\\
 Moscow State University \\
 Moscow, 119991, Russia\\
 and Institute for Information Transmission Problems RAS, Moscow, 127051, Russia
 }
\email{anaumov@cs.msu.su}

\author[V. Ulyanov]{V. Ulyanov}
\address{V. Ulyanov\\
 Faculty of Computational Mathematics and Cybernetics\\
 Moscow State University \\
 Moscow, 119991, Russia\\
 and National Research University Higher School of Economics (HSE),
 Moscow, 101000, Russia
 }
\email{vulyanov@cs.msu.su}
\thanks{F. G{\"o}tze was supported by CRC 701 ``Spectral Structures and Topological
Methods in Mathematics'', Bielefeld. A. Naumov and V. Ulyanov were supported by RSCF~14-11-00196.}


\dedicatory{Dedicated to Professor Yuri Prokhorov on the occasion
     of his $85$th birthday\\ (15.12.1929--16.07.2013)}

\date{\today}

\begin{abstract}
In this paper we consider asymptotic expansions for a class of sequences of symmetric functions of many variables.
Applications to classical and free probability theory are discussed.
\end{abstract}

\maketitle

\section{Introduction}

Most limit theorems such as the central limit theorem in finite dimensional and abstract spaces and the functional limit theorems
admit refinements in terms of asymptotic expansions in powers of $n^{-1/2}$, where  $n$ denotes the number of observations.
Results on asymptotic expansions of this type are summarized in many monographs, see for example~\cite{BhatRang2010}.

These expansions are obtained by very different techniques such as expanding the characteristic function of the particular statistic  for instance  in the case of linear statistics of independent observations, see e.g. Ch.~2, in~\cite{BhatRang2010} and  \S~1, Ch.~6, in~\cite{Petrov1975}.
 Other techniques  combine convolutions and characteristic functions to develop expansions for quadratic forms, see e.g.~\cite{GotzeZaitsev2014} and~\cite{ProkhUlyanov2013},
or for some discrete  distributions expansions are derived  starting   from a combinatorial formula for its distribution function, see e.g.~\cite{Gnedenko1961} and~\cite{Lauwerier1963}.
Alternatively  one might use  an expansion for an underlying empirical process and evaluate it on a domain defined by a functional
or statistic of this process. In those cases one would need to make approximations by Gaussian processes in suitable function spaces.

The aim of this paper is to show that for most of these expansions one could safely
ignore the underlying probability model and its ingredients  (like e.g.  proof of existence
of limiting processes and its properties).  
Indeed, similar expansions  can be derived in all of these models using a general scheme reflecting some (hidden) common features.
This is the universal collective  behavior caused by many independent asymptotically negligible variables influencing the distribution of  a functional.

The results of this paper may be considered as the extension of the results given by F.~G{\"o}tze in~\cite{Gotze1985} where
the following scheme of sequences of symmetric functions is studied.
Let  $h_n(\varepsilon, ... , \varepsilon_n), n \geq 1$  denote a sequence of  real functions defined on  $\R^{n}$ and suppose that the following conditions hold:
\begin{align}\label{cond_1}
&h_{n+1}(\varepsilon_1, ... , \varepsilon_{j}, 0, \varepsilon_{j+1}, ... , \varepsilon_n) = h_n(\varepsilon_1, ... , \varepsilon_j, \varepsilon_{j+1}, ... , \varepsilon_n);\\
&\frac{\partial}{ \partial \varepsilon_j} h_n(\varepsilon_1, ...,\varepsilon_{j}, ..., \varepsilon_n )\bigg|_{\varepsilon_j = 0} = 0 \text{ for all } j = 1, ... , n;\label{cond_2}\\
& h_n(\varepsilon_{\pi(1)}, ... ,\varepsilon_{\pi(n)}) = h_n(\varepsilon_1, ... , \varepsilon_n) \text{ for all } \pi \in S_n, \label{cond_3}
\end{align}
where  $S_n$ denotes the symmetric group.

 Consider the class of examples. Put
\begin{equation}\label{ex0}
h_n(\varepsilon_1,\ldots,\varepsilon_n)={\E_P} F(\varepsilon_1 (\delta_{X_1}-P)+ \ldots + \varepsilon_n(\delta_{X_n}-P)),
\end{equation}
where $F$ denotes a smooth functional defined on the space of signed measures and  $X_j$ denote random elements (in an arbitrary space) with common distribution $P$.
Thus $h_n$ is  the expected value of the functional $F$  of a {\it weighted} empirical process (based on the Dirac-measures in $X_1,\ldots, X_n$). Here property \eqref{cond_1} is obvious. Property \eqref{cond_2}, that is the locally quadratic dependence  on the weights  around zero, is  a consequence of the smoothness of $F$ and
the centering in $P$-mean of the signed random measures  $\delta_{X_J}-P$ and \eqref{cond_3} follows from the identical distribution of the$X_j$-s.
Properties \eqref{cond_1} and \eqref{cond_3} suggest to consider
the argument  $\epsilon_j$ of $h_n$ as  a weight  which controls the effect that   $X_j$ has on the distribution of the functional.
In \cite{Gotze1985}  the first author considered limits and expansions for functions  $h_n$ of equal weights $\varepsilon_j=n^{-1/2}, 1\le j\le n$ with applications to the case \eqref{ex0}. Those results explained the common structure of expansions for identical weights developed e.g. in~\cite{BhatRang2010} and \cite{GotzeHipp1978, Gotze1981, Gotze1989}.
In the following this scheme will be extended to the  case of {\it non identical} weights
$\varepsilon_j$, like in the class of examples where the functions $h_n$ are given by \eqref{ex0}. Moreover, the dependence of $F$ on the elements $X_j$ may be non linear.

Denote by $\varepsilon$ the $n$-vector $\varepsilon_j, 1\le j\le n$ and
by $\ve^d:=\sum_{j=1}^n\ve_j^d, d\ge 1$ the $d$-th power sum.
In the following we shall show that \eqref{cond_1}-\eqref{cond_3} ensure
the existence of a ''limit'' function
$h_{\infty}(\ve^2, \lambda_1,\ldots, \lambda_s)$ as a first
order approximation of $h_n$ together with ''Edgeworth-type'' asymptotic expansions, see e.g. the case of sums of non identically distributed random variables in Ch.6,~\cite{Petrov1975}.
These expansions are given in terms of polynomials of power sums  $\ve^d,d \ge 3 $. The coefficients  of these ''Edgeworth''-Polynomials, defined in \eqref{Edgeworthp} below, are given by derivatives of the limit function $h_{\infty}$
at  $\lambda_1=0, \ldots, \lambda_s=0$.

\begin{remark}[Algebraic Representations]
In case that $h_n$ is a {\it multivariate polynomial} of $\ve$, satisfying \eqref{cond_1}--\eqref{cond_3}, we may express it as polynomial in the algebraic base, $\ve^d, d \ge 1$, of symmetric power sums of $\ve$ with constant coefficients.
Note that
$$
\frac{\partial}{\partial_{\ve_j}}\ve^d\Big|_{\ve_j=0} = \delta_{d,1},
$$
where $\delta_{d,1}=1$, if $d=1$ and zero otherwise. Hence, \eqref{cond_2} entails  that in this representation $h_n$ is a polynomial of $\ve^d, d \ge 2$ only and does not depend on $\ve^1$.
Now we may write
$$
h_n(\ve)=P_{\ve^2}(\ve^3,\ldots, \ve^n  ),
$$
where $P_{\ve^2}$ denotes a polynomial with coefficients in the polynomial ring $\C[\ve^2]$ of the variable $\ve^2$.
Restricting ourselves to the sphere  $\ve^2=1$ for convenience, $P_{\ve^2}$ is the desired ''Edgeworth'' expansion, provided we introduce the following  grading of  monomials in the variables  $\ve^d, d \ge 3$ via $\Deg(\ve^d):= d-2$ and expand the polynomial $h_n$ in monomials of $\ve^3,\ldots, \ve^n$ by collecting terms according to this  grading.
\end{remark}

\subsection{Notations} Throughout the paper we will use the following notations.
We denote $\varepsilon^{d} := \sum_{i=1}^n \varepsilon_i^{d}$ and $|\varepsilon|^{d}: = \sum_{i=1}^n |\varepsilon_i|^{d}$.
Furthermore, we denote by $(\ve)_d$ and $|\ve|_d$ the $d$-th root of $\ve^{d}$ and $|\ve|^{d}$ respectively, i.e. $(\ve)_d : = (\ve^d)^{1/d}$ and $|\ve|_d: = (|\ve|^d)^{1/d}$. By $c$ or $C$ with indices or without we denote absolute constants, which can be different
in different places. Let $D^\alpha$, where $\alpha$ is a nonnegative integral vector, denote partial derivatives $\frac{\partial^{\alpha_1}}{\partial \ve_1^{\alpha_1}}...\frac{\partial^{\alpha_m}}{\partial \ve_m^{\alpha_m}}$, and
finally let $\alpha = \sum_{j=1}^m \alpha_j$.

\section{Results}

Introduce for an integer $s\geq0$ the functions,
\begin{equation}\label{limit function}
h_\infty(\lambda_1, ... , \lambda_s, \lambda): = \lim_{k \rightarrow \infty} h_{k+s} \left(\lambda_1, ... , \lambda_s, \frac{\lambda}{\sqrt{k}}, ... , \frac{\lambda}{\sqrt{k}}\right).
\end{equation}
Thus we consider the limit functions of $h_{k+s}$ as $k\rightarrow\infty$, where all but $s$ arguments are  equal, asymptotically negligible and taken from a $k-1$-sphere. In Theorems below we give sufficient conditions for the existence of the limits.
The following theorem is an analogue of the Berry-Esseen type inequality for sums of  non identically distributed independent random variables in probability theory, see e.g. Ch.6 in \cite{Petrov1975}.
\begin{theorem}\label{th_BE}
Assume that $h_n(\cdot), n \geq 1$, satisfies conditions~\eqref{cond_1}--\eqref{cond_3} and with some positive constant $B$ we have
\begin{align}\label{3dmomentcondition}
|D^{\alpha} h_n (\varepsilon_1, ... , \varepsilon_n)| \le B,
\end{align}
for all  $\varepsilon_1, ... , \varepsilon_n$,
and for all $\alpha = (\alpha_1, ... , \alpha_r)$ with $r \le 3$ such that
$$
\alpha_j \geq 2, \quad j = 1, ... , r, \quad \sum_{j = 1}^r (\alpha_j - 2) \le 1.
$$
Then there exists  $h_\infty(|\varepsilon|_2)$ defined by~\eqref{limit function} with $s=0$ and
$$
|h_n(\varepsilon_1, ... , \varepsilon_n) - h_\infty(|\varepsilon|_2)| \le c \cdot B \cdot \max (1, |\varepsilon|_2^3) |\varepsilon|^3,
$$
where $c$ is an absolute constant.
\end{theorem}
In case that $\varepsilon$  depends on $n$,  this theorem shows that if
\begin{equation}\label{thiedmomentconvergence}
\lim_{n \rightarrow \infty}|\varepsilon|_3 = 0
\end{equation}
then $h_n(\ve_1, ... , \ve_n)$ converges to the limit function $h_\infty(|\varepsilon|_2)$, which depends on $\varepsilon_1, ... , \varepsilon_n$ via  the $l_2$-norm $|\varepsilon|_2$ only.
This means that the sequence of symmetric functions (invariant with respect to $S_n$) may be approximated by a rotationally invariant function (invariant with respect to the orthogonal group $\mathcal O_n$).

Note though  that if~\eqref{thiedmomentconvergence} holds,  Theorem~\ref{th_BE}  doesn't provide an explicit formula for the function $h_{\infty}(|\varepsilon|_2)$, but guarantees its existence.

\begin{remark}
Investigating distributions of weighted sums, it has been shown in~\cite[Lemma~4.1]{KlarSodin2011}, that~\eqref{thiedmomentconvergence}  holds with high probability under the uniform measure,
 see as well inequality \eqref{KSbound} below.
\end{remark}
\begin{proof}[Proof of Theorem~\ref{th_BE}]
We divide the proof into three steps. In the first step we substitute each argument $\varepsilon_j$ by a block of the length $k$ of equal variables $\varepsilon_j / \sqrt{k}$.
This procedure doesn't change the $l_2$-norm $|\varepsilon|_2$. After $n$ steps we arrive at a function which depends  on $n \times k$ arguments
\begin{equation}\label{block function}
h_{nk}\left( \frac{\varepsilon_1}{\sqrt k}, ... , \frac{\varepsilon_1}{\sqrt k} , ... , \frac{\varepsilon_n}{\sqrt k}, ..., \frac{\varepsilon_n}{\sqrt k}  \right).
\end{equation}
We show that
\begin{align}\label{step_1}
&\left |h_n(\varepsilon_1, ... , \varepsilon_n) - h_{nk}\left( \frac{\varepsilon_1}{\sqrt k}, ... , \frac{\varepsilon_1}{\sqrt k} , ... , \frac{\varepsilon_n}{\sqrt k}, ..., \frac{\varepsilon_n}{\sqrt k}  \right)\right | \le c  \cdot B \cdot \max(1, |\ve|_2^3)|\ve|^3.
\end{align}
Hence this approximation step corresponds to Lindeberg's scheme of replacing  the
summands in the central limit theorem in probability theory by
corresponding  Gaussian random variables one by one (see, e.g., \cite{Lindeberg}
  and further development in \cite{Bergstroem}  and extension to
an invariance principle in \cite{Chaterjee}).
Here the replacement is performed not with a Gaussian variable but with a large block of equal weights of corresponding
$l_2$-norm. In the {\it second step},
still fixing $n$, we  determine the limit of the sequence of functions~\eqref{block function}, as $k$ goes to infinity.
We will show that in this case the limit depends on  $\varepsilon_1, ... , \varepsilon_n$, through its $l_2$-norm $|\varepsilon|_2$ only. It will be shown that
\begin{align}\label{step_2}
&\left |h_{nk}\left( \frac{\varepsilon_1}{\sqrt k}, ... , \frac{\varepsilon_1}{\sqrt k} , ... , \frac{\varepsilon_n}{\sqrt k}, ..., \frac{\varepsilon_n}{\sqrt k}  \right) -  h_{k}\left( \frac{|\varepsilon|_2}{\sqrt k}, ... , \frac{|\varepsilon|_2}{\sqrt k} \right )\right| \le c(\varepsilon) \cdot B \cdot k^{-1/2},
\end{align}
where $c(\ve)$ is some positive constant depending on $\ve$ only.

Finally, we may apply the arguments from Proposition~2.1 in~\cite{Gotze1985}. We show that there exists some function $h_\infty(|\ve|_2)$ such that
\begin{equation}\label{step_3}
\left |h_{k}\left( \frac{|\varepsilon|_2}{\sqrt k}, ... , \frac{|\varepsilon|_2}{\sqrt k} \right ) - h_\infty(|\varepsilon|_2)\right| \le c(\ve) \cdot B \cdot k^{-1/2}.
\end{equation}
From~\eqref{step_1}--~\eqref{step_3} it follows that
$$
|h_n(\varepsilon_1, ... , \varepsilon_n) - h_\infty(|\varepsilon|_2)| \le C \cdot B \cdot \max(1, |\ve|_2^3) |\varepsilon|^3 +  c(\ve) \cdot B \cdot  k^{-1/2} .
$$
Taking the limit  $k\to \infty$ we conclude the statement of the Theorem. In the following we shall provide the details for proof of the steps outlined above.

\noindent{\it First step}. We introduce additional notations. For simplicity let us denote $f_k(\delta_1, ... , \delta_k): = h_{n+k-1}(\delta_1, ... , \delta_k, \varepsilon_2, ... ,\varepsilon_n)$ and
\begin{align*}
&\underline \delta_{k} : = (\delta_1, ... , \delta_{k}) := \left (\frac{\varepsilon_1}{\sqrt k}, ... , \frac{\varepsilon_1}{\sqrt k} \right ) , \quad
\underline \delta_{k}^0 : = (\delta_1^0, ..., \delta_k^0):= (\varepsilon_1, 0, ... , 0).
\end{align*}
Using Taylor's formula we may write
\begin{align*}
f_{k}(\underline \delta_k) - f_{k}(\underline \delta_k^0) = \sum_{j=1}^k \frac{\partial f_{k}(\underline \delta_k^0)}{\partial \delta_j} (\delta_j - \delta_j^0) +\frac{1}{2} \sum_{j,l = 1}^k \frac{\partial^2 f_{k}(\underline \delta_k^0)}{\partial \delta_j \partial \delta_l} (\delta_j - \delta_j^0)(\delta_l - \delta_l^0) +R_{31},
\end{align*}
where $R_{31}$ is a remainder term which will be estimated later. In what follows we shall denote by $R_{3i}$, for some $i \in \mathbb N,$ the remainder terms in Taylor's expansion.
By~\eqref{cond_1} all summands in the first sum equals zero except for $j=1$. Consider the second sum. If $j \neq l$ and $j, l \neq 1$ then the corresponding summand equals zero.
Condition~\eqref{cond_2} yields
$$
\frac{\partial}{\partial \delta_j} \frac{\partial}{\partial \delta_l} f_{k}(\underline \delta_k^0)= 0
$$
provided that $j \neq l$, $j, l \neq 1$ and
$$
\sum^k_{l=2}\frac{\partial}{\partial \delta_1} \frac{\partial}{\partial \delta_l} f_{k}(\underline \delta_k^0)= R_{32}.
$$
for all $l = 2, ... , k$.
Expanding the non zero terms in the first and second sum we obtain
\begin{align*}
&\frac{\partial}{\partial \delta_1} f_{k}(\underline \delta_k^0) = \frac{\partial^2}{\partial \delta_1^2} f_{k}(\underline \delta_k^0)\big|_{\delta_1^0 = 0}  \delta_1^0 + R_{33}, \\
&\frac{\partial^2}{\partial \delta_j^2}f_{k}(\underline \delta_k^0) = \frac{\partial^2}{\partial \delta_1^2}f_{k}(\underline \delta_k^0)\big|_{\delta_1^0 = 0} + R_{34}.
\end{align*}
Applying condition~\eqref{cond_3} we may sum the coefficients of the second derivatives of $f_{k}$ and get
\begin{equation}\label{eq: variance}
\varepsilon_1 \left (\frac{\varepsilon_1}{\sqrt k} - \varepsilon_1 \right ) + \frac{1}{2} \left (\frac{\varepsilon_1}{\sqrt k} - \varepsilon_1 \right )^2 + \frac{k-1}{2} \left (\frac{\varepsilon_1}{\sqrt k} \right )^2 = 0.
\end{equation}
It remains to investigate the terms $R_{3l}, l = 1, ... ,4$, and show that
$$|R_{3l}| \le C \cdot B \cdot (|\varepsilon_1|^3 + |\varepsilon_1|^4 +|\varepsilon_1|^6).$$
Let us consider $R_{31}$. First we note that $R_{31}$ is
the sum of the third derivatives of $f_k$ at some intermediate point $\hat{\underline \delta}_k^0$:
$$
\sum_{j,l,m = 1}^k \frac{\partial^3}{\partial \delta_j \partial \delta_l \partial \delta_m}f_k(\hat{\underline \delta}_k^0)(\delta_j - \hat \delta_j^0)(\delta_l - \hat \delta_l^0)(\delta_m - \hat \delta_m^0).
$$
If the partial derivative with respect to $\delta_j$ (or $\delta_l, \delta_m$) is of order one we need to add an additional expansion  with respect to this variable around zero using~\eqref{cond_2}.
In this way, we finally  get
\begin{align*}
&\left| \sum_{j,l,m = 1}^k \frac{\partial^3}{\partial \delta_j \partial \delta_l \partial \delta_m}f_k(\hat{\underline \delta}_k^0)(\delta_j - \hat \delta_j^0)(\delta_l - \hat \delta_l^0)(\delta_m - \hat \delta_m^0)\right| \le C \cdot B \cdot (|\varepsilon_1|^3 + |\varepsilon_1|^4 +|\varepsilon_1|^6).
\end{align*}
The other terms, that is $R_{3l}, l = 2, 3, 4$, may be treated in a similar way.Repeating  this procedure $n - 1$ times we arrive at the function~\eqref{block function} and the bound~\eqref{step_1}.
\vskip 0.1cm

\noindent{\it Second step}. The proof is similar to the previous step. Here  $n$ is fixed and we derive bounds in terms  of powers of $k^{-1/2}$. Applying assumption~\eqref{cond_3} we may rearrange the arguments in $h_{nk}(\cdot)$ and get
$$
h_{nk}\left( \frac{\varepsilon_1}{\sqrt k}, ... , \frac{\varepsilon_1}{\sqrt k} , ... , \frac{\varepsilon_n}{\sqrt k}, ..., \frac{\varepsilon_n}{\sqrt k}  \right) =
h_{nk}\left( \frac{\varepsilon_1}{\sqrt k}, ... , \frac{\varepsilon_n}{\sqrt k} , ... , \frac{\varepsilon_1}{\sqrt k}, ..., \frac{\varepsilon_n}{\sqrt k}  \right).
$$
Let us denote 
$$
f_n(\delta_1, ..., \delta_n) : = h_{nk}\left(\delta_1, ..., \delta_n, \frac{\varepsilon_1}{\sqrt k}, ... , \frac{\varepsilon_n}{\sqrt k} , ... , \frac{\varepsilon_1}{\sqrt k}, ..., \frac{\varepsilon_n}{\sqrt k}  \right)
$$
and choose the following argument vectors
\begin{align*}
&\underline \delta_{n} : = (\delta_1, ... , \delta_{n}) := \left (\frac{\varepsilon_1}{\sqrt k}, ... , \frac{\varepsilon_n}{\sqrt k} \right ) , \quad \underline \delta_{n}^0 : = (\delta_1^0, ..., \delta_n^0):= \left(\frac{|\varepsilon|_2}{\sqrt k}, 0, ... , 0\right).
\end{align*}
We shall  estimate  $f_n(\underline\delta_{n}) - f_n(\underline\delta_{n}^0)$ by repeating the same arguments as in the first step. We omit the details, but would like mention  that instead of~\eqref{eq: variance} we shall use here that
 \begin{equation*}
\frac{|\varepsilon|_2}{\sqrt{k}} \left (\frac{\varepsilon_1}{\sqrt k} - \frac{|\varepsilon|_2}{\sqrt k} \right ) + \frac{1}{2} \left (\frac{\varepsilon_1}{\sqrt k} - \frac{|\varepsilon|_2}{\sqrt k} \right )^2 + \frac{1}{2} \sum_{j=2}^n \left (\frac{\varepsilon_j}{\sqrt k} \right )^2 = 0.
\end{equation*}
Thus, we finally arrive at the following bound
$$
|f_n(\delta_{n}) - f_n(\delta_{n}^0)| \le c(\varepsilon) \cdot B \cdot k^{-\frac32}.
$$
Repeating this procedure $k - 1$ times we obtain the bound~\eqref{step_2}.\\

\noindent{\it Third step}. We consider the difference of the value of $h_k$ at the point
$$
\underline \varepsilon_k = (|\varepsilon|_2 k^{-1/2}, ... , |\varepsilon|_2 k^{-1/2})
$$
and the value of $h_{k+r}$ at the point
$$
\underline \varepsilon_{k+r} = (|\varepsilon|_2 (k+r)^{-1/2}, ... ,|\varepsilon|_2 (k+r)^{-1/2}).
$$
We show, similar to the arguments in the previous steps,  that (compare as well the proof in~\cite{Gotze1985}[Proposition~2.1]),
$$
|h_{k}(\underline \varepsilon_k) - h_{k+r}(\underline \varepsilon_{k+r})| \le c(\varepsilon) \cdot B \cdot \sum_{p=k}^{k+r-1} p^{-3/2}.
$$
Thus, $h_k(\underline \varepsilon_k)$ is a Cauchy sequence in $k$ with a limit, say $h_\infty(|\varepsilon|_2)$. This fact concludes the proof of the theorem.
\end{proof}

To formulate the asymptotic expansion of the function $h_n(\cdot), n \geq 1$, we have to introduce additional notations.  We introduce the following differential operators by means of formal power series identities.
Define cumulant differential operators $\kappa_p(D)$ by means of
$$
\sum_{p=2}^{\infty} p!^{-1} \varepsilon^p \kappa_p(D) = \ln \left(1 + \sum_{p=2}^\infty p!^{-1} \varepsilon^p D^p \right )
$$
in the formal variable $\varepsilon$. One may easily compute the first cumulants. For example, $\kappa_2 = D^2, \kappa_3 = D^3, \kappa_4 = D^4 - 3 D^2 D^2$.
Define Edgeworth polynomials by means of the following formal series in $\kappa_r, \tau_r$ and the formal variable $\varepsilon$
$$
\sum_{r=0}^{\infty} \varepsilon^r P_r(\tau_{*} \kappa_{*}) = \exp \left (\sum_{r = 3}^{\infty} r!^{-1} \varepsilon^{r-2} \kappa_r \tau_r \right )
$$
which yields
\begin{align}\label{Edgeworthp}
P_r(\tau_{*}\kappa_{*}) = &\sum_{m = 1}^r m!^{-1} \sum_{j_1, ... , j_m} (j_1 + 2)!^{-1} \tau_{j_1 + 2} \kappa_{j_1+2} \\
&\times (j_2 + 2)!^{-1} \tau_{j_2 + 2} \kappa_{j_2+2} ... (j_m + 2)!^{-1} \tau_{j_m + 2} \kappa_{j_m+2},\nonumber
\end{align}
where the sum $\sum_{j_1, ... , j_m}$ extends over all $m$-tuples of positive integers $(j_1, ... , j_m)$ satisfying
$\sum_{q = 1}^m j_q = r$ and $\kappa_{*} = (\kappa_3, ... , \kappa_{r+2}), \tau_{*} = (\tau_3, ... , \tau_{r+2})$.
For example,
\begin{align}
&P_1(\tau_{*} \kappa_{*}) = \frac{1}{6} \tau_3 \kappa_3 = \frac{1}{6} \tau_3^3 D^3, \nonumber \\
\label{P2}
&P_2(\tau_{*} \kappa_{*} ) = \frac{1}{24} \tau_4 \kappa_4 + \frac{1}{72} \tau_3^2 \kappa_3 \kappa_3 = \frac{1}{24} \tau_4 (D^4 - 3 D^2 D^2) + \frac{1}{72} \tau_3^2 D^3 D^3.
\end{align}

In the following theorem we will assume that $\varepsilon$ is a vector on the unit sphere, i.e. $|\varepsilon|_2 = 1$. It is also possible to consider the general case $|\varepsilon|_2 = r, r > 1$,
but then the remainder terms will have a more difficult structure. In what follows we shall drop the dependence of $h_{\infty}$ on  the argument $|\varepsilon|_2$ in the notation of this function.
\begin{theorem}\label{asymptotic_expansion}
Assume that $h_n(\varepsilon_1,..., \varepsilon_n), n \geq 1$, satisfies conditions~\eqref{cond_1}, \eqref{cond_2} and \eqref{cond_3} together with $|\varepsilon|_2 = 1$.
Suppose that
\begin{align}
|D^{\alpha} h_n (\varepsilon_1, ... , \varepsilon_n)| \le B,
\end{align}
for all  $\varepsilon_1, ... , \varepsilon_n$, where $B$ denotes some positive constant, $\alpha = (\alpha_1, ... , \alpha_r), r \le s$, and
$$
\alpha_j \geq 2, \quad j = 1, ... , r, \quad \sum_{j = 1}^r (\alpha_j - 2) \le s -2.
$$
Then
$$
h_n(\varepsilon_1, ... , \varepsilon_n) = h_\infty +
\sum_{l=1}^{s-3} P_l(\varepsilon^{*} \kappa_{*}) h_\infty(\lambda_1, ... , \lambda_l)\big|_{\lambda_1 = ... =\lambda_l =0 } + R_s,
$$
where $P_l(\varepsilon^{*} \kappa_{*})$ is defined in~\eqref{Edgeworthp} with
$\varepsilon^{*} =  (\varepsilon^3, \dots , \varepsilon^{l+2}), \kappa_{*} = (\kappa_3, ... , \kappa_{l+2})$ and
$$
|R_s| \le c_s \cdot B \cdot |\varepsilon|^{s}.
$$
with some absolute constant $c_s$.
\end{theorem}

As an example consider the case $s = 5$. Then by~\eqref{P2}
\begin{align}\label{eq: expansion5}
&h_n(\varepsilon_1, ... , \varepsilon_n)= h_\infty + \frac{\ve^3}{6} \frac{\partial^3}{\partial \lambda^3} h_\infty(\lambda)\big|_{\lambda=0}  \\
&+\left [ \frac{\varepsilon^4}{24} \left ( \frac{\partial^4}{\partial \lambda_1^4} - 3 \frac{\partial^2}{\partial \lambda_1^2}\frac{\partial^2}{\partial \lambda_2^2} \right )
+ \frac{(\varepsilon^3)^2}{72}\frac{\partial^3}{\partial \lambda_1^3}\frac{\partial^3}{\partial \lambda_2^3} \right ] h_\infty(\lambda_1, \lambda_2)\big|_{\lambda_1=0, \lambda_2 = 0} + \mathcal O(|\varepsilon|^5). \nonumber
\end{align}

Before we start  proving  Theorem~\ref{asymptotic_expansion} we have to introduce one more notation. For any sequence $\tau_p, p \geq 1$, of
formal variables define $\tilde P(\tau_{*} \kappa_{*})$ as a polynomial in the cumulant operators $\kappa_p$ multiplied by $\tau_p$ by the following formal power series in $\mu$:
\begin{equation}\label{tildeP}
\sum_{j=0}^{\infty} \tilde P_j(\tau_{*}\kappa_{*}) \mu^j : = \exp \left(\sum_{j=2}^{\infty} j!^{-1} \tau_j \kappa_j(D) \mu^j \right).
 \end{equation}
For example, $\tilde P_0 = 1, \tilde P_1 = 0, \tilde P_2 = \frac{1}{2} \tau_2 D^2, \tilde P_3 = \frac{1}{6} \tau_3 D^3$, and
$$
\tilde P_4 = \frac{1}{24} \tau_4 (D^4 - 3 D^2 D^2) + \frac{1}{8} \tau_2^2 D^2 D^2.
$$
If $\tau_p = \tau^p, p \geq 1$ then
\begin{equation}\label{eq: tilde P relations 1}
\tilde P_j = j!^{-1} \tau_j D^j.
\end{equation}
Furthermore, one may verify that functional identities of $\exp$ in \eqref{tildeP} yield the identities
\begin{equation}\label{eq: tilde P relations 2}
\sum_{j+l = r} \tilde P_j (\tau_{*} \kappa_{*})\tilde P_j (\tau_{*}' \kappa_{*}) = \tilde P_r ((\tau_{*} + \tau_{*}') \kappa_{*}) .
\end{equation}
There is a relation between the Edgeworth polynomials $P_r(\cdot)$ and $\tilde P_r(\cdot)$ which may be expressed  in the following relation.
\begin{statement} We have
\begin{equation}\label{relationbetweenpolinomial}
\sum_{r=1}^{\infty} [P_r(\tau_{*} \kappa_{*})]_l = \sum_{r=1}^l \tilde P_r(\tau_{*} \kappa_{*}),
\end{equation}
where $[\cdot]_l$  denotes the sum of all monomials $\tau_{1}^{p_1} \cdot \cdot \cdot \tau_{r+2}^{p_r + 2}$ in $P_r(\tau_{*} \kappa_{*})$ such that $p_1 + 2p_2 + ... + (r+2) p_{r+2} \le l$.
\end{statement}
We shall use~\eqref{relationbetweenpolinomial} in the proof of Theorem~\ref{asymptotic_expansion}. The following Lemma allows us to rewrite the derivatives of $h_n(\varepsilon_1, ... , \varepsilon_n)$
via
derivatives in additional variables using the definition of $\tilde P_r$.
\begin{lemma}\label{l_polynomials}
Suppose that the conditions~\eqref{cond_1}, \eqref{cond_2} and \eqref{cond_3}  hold. Then
\begin{align*}
&\sum_{j=2}^m \frac{1}{j!} \frac{\partial}{\partial \varepsilon^j} h_n(\varepsilon, \varepsilon_2, ... , \varepsilon_n)\big|_{\varepsilon = 0} (\eta^j - \varepsilon^j) \\
&=\sum_{r=2}^m \tilde P_r((\eta^{*} - \varepsilon^{*})\kappa_{*}(D))h_{n+m}  (\lambda_1, .. , \lambda_k, \varepsilon, \varepsilon_2, ... , \varepsilon_n ) \big |_{\lambda_1 = ... = \lambda_m = 0} + \mathcal O \left (\varepsilon^{m} \right ).
\end{align*}
\end{lemma}
\begin{proof}
This relation  is a consequence of the following simple computations:
\begin{align*}
&\sum_{r = 1}^m \tilde P_r((\eta^{*} - \varepsilon^{*}) \kappa_{*}(D)) h_{m+n} (\underbrace{0, ... , 0}_\text{$m$}, \varepsilon, \varepsilon_2, ... , \varepsilon_n) \\
&\overset{\eqref{eq: tilde P relations 1}}{=}\sum_{j=1}^m \sum_{l+r = j \atop r \geq 1} \tilde P_r((\eta^{*} - \varepsilon^{*}) \kappa_{*}(D)) \tilde P_r(\varepsilon^{*} \kappa_{*}(D)) \\
&\qquad\qquad\qquad\times h_{m+n} (\underbrace{0, ... , 0}_\text{$m+1$}, \varepsilon_2, ... , \varepsilon_n) + \mathcal O \left (\varepsilon^{m} \right ). \\
& \overset{\eqref{eq: tilde P relations 2}}{=} \sum_{j=1}^m (\tilde P_r(\eta^{*} \kappa_{*}(D)) -  \tilde P_r(\varepsilon^{*} \kappa_{*}(D))  h_{m+k} (\underbrace{0, ... , 0}_\text{$m+1$}, \varepsilon_2, ... , \varepsilon_n) + \mathcal O \left (\varepsilon^{m} \right ). \\
&\overset{\eqref{eq: tilde P relations 1}}{=}\sum_{j=2}^m \frac{1}{j!} \frac{\partial}{\partial \varepsilon^j} h_n(\varepsilon, \varepsilon_2, ... , \varepsilon_n)\big|_{\varepsilon = 0} (\eta^j - \varepsilon^j) + \mathcal O \left (\varepsilon^{m} \right ).
\end{align*}
\end{proof}
\begin{proof}[Proof of Theorem~2.2]
We prove this theorem by  induction on the length of the expansion. Consider the difference
$$
h_n(\varepsilon_1, ... , \varepsilon_n) - h_{nk}\left( \frac{\varepsilon_1}{\sqrt k}, ... , \frac{\varepsilon_1}{\sqrt k} , ... , \frac{\varepsilon_n}{\sqrt k}, ..., \frac{\varepsilon_n}{\sqrt k}  \right)
$$
and divide it into a sum of $n$ terms, replacing  successively the argument $\varepsilon_j, j=1,\ldots,n$ by the $k$-vector $(k^{-1/2} \varepsilon_j, ... , k^{-1/2} \varepsilon_j)$, similar  to the arguments used in the previous Theorem~\ref{th_BE}.
We start with the case $j=1$
and denote $h_k(\delta_1, ... , \delta_k): = h_{n+k-1}(\delta_1, ... , \delta_k, \varepsilon_2, ... ,\varepsilon_n)$. Set
\begin{align*}
&\underline \delta_{k} : = (\delta_1, ... , \delta_{k}) := \left (\frac{\varepsilon_1}{\sqrt k}, ... , \frac{\varepsilon_1}{\sqrt k} \right ), \quad \underline \delta_{k}^0 : = (\delta_1^0, ..., \delta_k^0):= (\varepsilon_1, 0, ... , 0).
\end{align*}
A Taylor expansion  yields
\begin{align}\label{fistroundofTS}
h_{k}(\underline \delta_k) - h_{k}(\underline \delta_k^0) = \sum_{0 < |\alpha| < s} \alpha!^{-1}D^{\alpha} h_{k}(\underline \delta_k^0)((\underline \delta_k - \underline \delta_k^0)^\alpha + R_{1,s}, \nonumber
\end{align}
with a remainder term $R_{1,s}$ such that  $|R_{1,s}| \le c_s \cdot B \cdot |\varepsilon_1|^{s}$. To simplify our notations we introduce  the following convention.  By $R_l, l \geq 1$ we shall denote a remainder term which is of order
$\mathcal O(|\varepsilon|^l)$ omitting the explicit dependence on $B$. By $R_{1,l}, l \geq 1$ we shall denote a remainder term of order $\mathcal O(|\varepsilon_1|^l)$.\\
In order to use  condition~\eqref{cond_2} we expand each derivative $D^{\alpha} h_{k}(\underline \delta_k^0)$, where $\alpha = (\alpha_{j_1}, ... , \alpha_{j_p}), 1 \le j_1 \le ... \le j_p \le k$,
around $\delta_{j_r}=0, r = 1, ... , p$. This yields
\begin{align}
D^{\alpha} h_{k}(\underline \delta_k^0) = \sum_{0 < |\alpha|+|\beta| < s} \beta!^{-1} D^{\alpha+\beta} h_{k}(\underline \delta_k^0)(\underline \delta_k^0)^\beta + R_{1,s},
\end{align}
The binomial formula implies
$$
\sum_{j+k=r, j \geq 1} \frac{1}{j!\cdot k!}(\varepsilon - \eta)^j \eta^k = \frac{1}{r!}(\varepsilon^r - \eta^r).
$$
Applying this relation to~\eqref{fistroundofTS} we get
\begin{align*}
h_{k}(\underline \delta_k) - h_{k}(\underline \delta_k^0) =\sum_{0 < |\gamma| < s} \gamma!^{-1} D^{\gamma} h_k(0, ... , 0) \prod_{i=1}^k [\delta_i^{\gamma_i} - (\delta_i^0)^{\gamma_i} ] + R_{1,s}.
\end{align*}
In order to use the induction assumption which will be formulated later in terms of the function values $h_k(\underline \delta_k)$, we have to use expansions and derivatives in  additional variables at zero.

Introduce for $j = 1, \dots, k$ and $p = 3, \dots, s-1,$
$$
\Delta_j^p := \delta_j^p - \left(\delta_j^0 \right)^p \,\, \text{and} \,\, \Delta_j^* := (\Delta_j^3, \dots, \Delta_j^{s-1}).
$$
Applying Lemma~\ref{l_polynomials} we get
\begin{align}\label{finaltaylor}
&h_{k}(\underline \delta_k) - h_{k}(\underline \delta_k^0) =\sum \tilde P_{r_1}(\Delta_{j_1}^{*} \kappa_{*}) \cdot \cdot \cdot \tilde P_{r_m}(\Delta_{j_m}^{*} \kappa_{*}) h_{k + s}(\underline \delta_k, 0, ... , 0) + R_{1,s},
\end{align}
where the sum extends over all combination of $r_1, ... , r_m \geq 2, m = 1, 2, ... $, such that $r_1 + ... + r_m < s$ and all ordered $m$ - tuples of positive induces $1 \le j_r \le m$ without repetition.
Assume that for $l = 3, ... , s-1$ we have already proved  that
\begin{align}\label{induction_step}
D^{\alpha} h_n(\varepsilon_1, ... , \varepsilon_n) = \sum_{j = 0}^{l-3} P_j(\ve^{*} \kappa_{*}) h_{\infty} + R_{1,l}.
\end{align}
We start with the case $l = 3$, which follows from Theorem~\ref{th_BE}.  Applying the induction assumption~\eqref{induction_step} to~\eqref{finaltaylor} we get
\begin{align*}
h_{k}(\underline \delta_k) - h_{k}(\underline \delta_k^0)= \sum \tilde P_{r_1}(\Delta_{j_1}^{*} \kappa_{*}) \cdot \cdot \cdot \tilde P_{r_k}(\Delta_{j_k}^{*} \kappa_{*}) P_{r_0}(\ve^{*} \kappa_{*} ) h_{\infty} + R_{1,s},
\end{align*}
where the sum extends over all indices $r_1, ... , r_m \geq 1, r_0 \geq 0$, such that $r_0 + r_1 + ... + r_m \le s$. We shall rewrite this relation as follows
\begin{align}\label{square_brack}
h_{k}(\underline \delta_k) - h_{k}(\underline \delta_k^0) = \sum_{r_0 = 0}^{s- 4} P_{r_0}(\ve^{*} \kappa_{*} ) \sum_{j} \left [\prod_{l=1}^{s-r_0} \left [ \sum_{v_l = 1}^s  \tilde P_{v_l}(\Delta_{j_l}^{*} \kappa_{*}) \right ] \right ]_{s-r_0} h_{\infty} + R_{1,s},
\end{align}
where $[ \,\, ]_r$ denotes all terms of the enclosed formal power series which are proportional to monomials $\Delta_{j_1}^{r_1}\dots \Delta_{j_1}^{r_1}$ with $r_1 + \dots + r_k \leq r$ and $k \leq m$. The right hand side of ~\eqref{square_brack} may be expressed as
\begin{align}\label{exp1}
\sum_{r_0 = 0}^{s- 4} P_{r_0}(\ve^{*} \kappa_{*}) \left [ \exp \left [ \sum_{p=2}^\infty \left ( \sum_{ j = 1}^k \Delta_j^p \right ) p!^{-1} \kappa_p \right ] - 1 \right ]_{s-r_0} h_\infty + R_{1,s}
\end{align}
It is easy to see that
$$
\sum_{ j = 1}^k \Delta_j^2  = \varepsilon_1^2 \left [ \left( \frac{1}{k} - 1 \right ) + \frac{k-1}{k} \right ] = 0
$$
and
\begin{equation}\label{Delta_P}
\sum_{ j = 1}^k \Delta_j^p  = \varepsilon_1^p \left [ \left( \frac{1}{k^{p/2}} - 1 \right ) + \frac{k-1}{k^{p/2}} \right ] = -\varepsilon_1^p + \mathcal O \left (\frac{|\varepsilon_1|^p}{k^{p/2}} \right), \quad p > 2.
\end{equation}
Using~\eqref{relationbetweenpolinomial} we get
\begin{align}
&P_{r_0}(\ve^{*} \kappa_{*}) \left [ \exp \left [ \sum_{p=2}^\infty \left ( \sum_{ j = 1}^k \Delta_j^p \right ) p!^{-1} \kappa_p \right ] - 1 \right ]_{s-r_0} h_\infty  \nonumber\\
&=P_{r_0}(\ve^{*} \kappa_{*}) \sum_{i = 3}^{s-r_0-1} \tilde P_i \left (\left ( \sum_{ j = 1}^k \Delta_j^{*}\right )\kappa_{*} \right ) h_{\infty} \nonumber\\
\label{exp2}
&= P_{r_0}(\ve^{*} \kappa_{*}) \sum_{i = 1}^{\infty} \left [ P_i \left ( \left ( \sum_{ j = 1}^k \Delta_j^{*} \right ) \kappa_{*} \right ) \right ]_{s-r_0-1} h_{\infty}
\end{align}
and
\begin{align}\label{deletebracket}
P_{r_0}(\ve^{*} \kappa_{*})  P_r \left ( \sum_{j=1}^k \Delta_j^{*} \kappa_{*} \right ) h_{\infty} = P_{r_0}(\ve^{*} \kappa_{*}) \left [P_r \left ( \sum_{j=1}^k \Delta_j^{*} \kappa_{*} \right ) \right ]_{s-r_0-1} h_{\infty} + R_{1,s}
\end{align}
By~\eqref{exp2}--\eqref{deletebracket} we may rewrite~\eqref{exp1} in the following way
$$
h_{k}(\underline \delta_k)- h_{k}(\underline \delta_k^0) = \sum_{r_0 = 0}^{s-4} P_{r_0}(\ve^{*} \kappa_{*}) \sum_{r = 1}^{s-3-r_0} P_r\left ( \sum_{j=1}^k \Delta_j^{*} \kappa_{*} \right ) h_{\infty} + R_{1,s}.
$$
Since
$$
\sum_{r+q = k} P_r(\tau_{*} \kappa_{*}) P_q(\tau_{*}' \kappa_{*}) = P_k((\tau_{*} + \tau_{*}') \kappa_{*}), \quad q,r,k \geq 0.
$$
we conclude
\begin{align*}
h_{k}(\underline \delta_k)- h_{k}(\underline \delta_k^0) = \sum_{r = 1}^{s-3} \left [ P_r ((\ve^{*} - \varepsilon_1^{*}) \kappa_{*} ) - P_r(\ve^{*} \kappa_{*}) \right ]h_{\infty}(\lambda_1, ... , \lambda_s)\big|_{\lambda_1 = ... = \lambda_s = 0} + R_{1,s},
\end{align*}
where we have used~\eqref{Delta_P}.

Thus, we replaced $\varepsilon_1$  by $(\varepsilon_1/\sqrt{k}, \dots, \varepsilon_1/\sqrt{k})$ and found a corresponding expansion. We shall now repeat the same procedure for the remaining  $\varepsilon_j$.

For all $j \geq 2$ we replace  $\varepsilon_j$ by $(k^{-1/2} \varepsilon_j, ... , k^{-1/2} \varepsilon_j)$.
It is easy to see that replacing $h_n(\varepsilon_1, ... , \varepsilon_n)$ by  $h_{n+k-1}(\underline \delta_{k}, \varepsilon_2, ... ,\varepsilon_n)$ with
$$
\underline \delta_{k} : = (\delta_1, ... , \delta_{k}) = \left (\frac{\varepsilon_1}{\sqrt k}, ... , \frac{\varepsilon_1}{\sqrt k} \right )
$$
we can use~\eqref{induction_step} with $(\underline \delta_{k}, \varepsilon_{[2:n]})$ instead of
  $\ve$, where $\varepsilon_{[2:n]}: = (\varepsilon_2, ... , \varepsilon_n)$. The function $h_{\infty}$
will be the same  in both expansions since it depends  on  $|\varepsilon|_2$ only  and
$|\underline \delta_k|^2 + |\varepsilon_{[2:n]}|^2 = |\varepsilon|_2$. The same arguments  may be applied for all $j \geq 2$. Hence, repeating this procedure $n-1$ times we get
\begin{align*}
&h_{nk}\left( \frac{\varepsilon_1}{\sqrt k}, ... , \frac{\varepsilon_1}{\sqrt k} , ... , \frac{\varepsilon_n}{\sqrt k}, ..., \frac{\varepsilon_n}{\sqrt k}  \right) - h_{k}(\varepsilon_1, ... , \varepsilon_n) \\
&= \sum_{j=1}^n\sum_{r = 1}^{s-3} \left [ P_r((\varepsilon_{[j:n]}^{*} - \varepsilon_j^{*}) \kappa_{*}) - P_r(\varepsilon_{[j:n]}^{*} \kappa_{*}) \right ]h_{\infty}(\lambda_1, ... , \lambda_s)\big|_{\lambda_1 = ... = \lambda_s = 0} + R_s,
\end{align*}
where $\varepsilon_{[j:n]}:=(\varepsilon_{j}, ... ,\varepsilon_n)$.

To finish the proof we need to show that for fixed $r$ the following identity holds
\begin{align}\label{eq: finish of the proof}
\sum_{j=1}^n \left [ P_r((\varepsilon_{[j:n]}^{*}  - \varepsilon_j^{*}) \kappa_{*}) - P_r(\varepsilon_{[j:n]}^{*} \kappa_{*}) \right ] = - P_r(\varepsilon^{*} \kappa_{*}).
\end{align}
The relation~\eqref{eq: finish of the proof} follows from the following simple observation. Let $m \geq 1$ be a fixed integer and $(j_1, ... , j_m)$ be a vector of positive numbers such that $j_1 + ... + j_m = r$. Then
\begin{align*}
&\sum_{i=1}^n (\varepsilon_{[i:n]}^{j_1 + 2} - \varepsilon_i^{j_1+2})\cdot \cdot \cdot ({\varepsilon_{[i:n]}}^{j_m + 2} - \varepsilon_i^{j_m+2})
-\sum_{i=1}^n \varepsilon_{[i:n]}^{j_1 + 2} \cdot \cdot \cdot {\varepsilon_{[i:n]}}^{j_m + 2}\\
& = -{\varepsilon}^{j_1+2} \cdot \cdot \cdot  {\varepsilon}^{j_m+2}.
\end{align*}
For a proof it is enough to note that for all $i \geq 1$
$$
(\varepsilon_{[i:n]}^{j_1 + 2} - \varepsilon_i^{j_1+2})\cdot \cdot \cdot ({\varepsilon_{[i:n]}}^{j_m + 2} - \varepsilon_i^{j_m+2})=
{\varepsilon_{[i+1:n]}}^{j_1 + 2} \cdot \cdot \cdot {\varepsilon_{[i+1:n]}}^{j_m + 2}.
$$
Applying~\eqref{eq: finish of the proof} we arrive at
\begin{align*}
&h_{n}(\varepsilon_1, ... , \varepsilon_n) - h_{nk}\left( \frac{\varepsilon_1}{\sqrt k}, ... , \frac{\varepsilon_1}{\sqrt k} , ... , \frac{\varepsilon_n}{\sqrt k}, ..., \frac{\varepsilon_n}{\sqrt k}  \right) \\
&= \sum_{r = 1}^{s-3} P_r(\ve^{*} \kappa_{*})h_{\infty}(\lambda_1, ... , \lambda_s)\big|_{\lambda_1 = ... = \lambda_s = 0} + R_s.
\end{align*}
Repeating now the last two steps in the proof of the previous Theorem~\ref{th_BE} and taking the limit  $k \rightarrow \infty$ we get
\begin{align*}
h_{n}(\varepsilon_1, ... , \varepsilon_n) - h_{\infty} = \sum_{r = 1}^{s-3} P_r(\ve^{*} \kappa_{*})h_{\infty}(\lambda_1, ... , \lambda_s)\big|_{\lambda_1 = ... = \lambda_s = 0} + R_s.
\end{align*}
This proves~\eqref{induction_step} for $l = s$ and $\alpha = 0$. Hence, the induction is completed and the Theorem is proved.
\end{proof}

\section{Application of Theorem~\ref{asymptotic_expansion}}
In this section we illustrate in some examples how one may apply Theorem~\ref{asymptotic_expansion} to derive an asymptotic expansion of various functions in probability theory.
\subsection{Expansion in the Central Limit Theorem for Weighted Sums}
As the first example let us consider the sequence of independent random variables $X, X_j, j \in \mathbb N$, taking values in $\R$ with a common distribution function $F$. Suppose the $\E X = 0, \E X^2 = 1$. Consider the weighted sum $S_{\varepsilon} = \varepsilon_1 X_1 + ... + \varepsilon_n X_n$. As $h_n$ we may choose the characteristic function of $S_{\varepsilon}$, i.e.,
$$
h_n(\varepsilon_1, ... , \varepsilon_n) = \E e^{i t (\varepsilon_1 X_1 + ... + \varepsilon_n X_n)}.
$$
From Theorem~\ref{th_BE} we know that $h_\infty(|\varepsilon|_2)$ exists provided that the condition~\eqref{3dmomentcondition} holds.
In our setting this condition holds when $\E |X|^3 < \infty$ . It is well known, see e.g. Ch.5 in \cite{Petrov1975}, that
$$
h_\infty(|\varepsilon|_2) = \E e^{i t G},
$$
where $G \sim N(0, |\varepsilon|_2)$. In what follows we shall assume  $|\varepsilon|_2 = 1$. The rate of convergence is given by $|\varepsilon|_3^3$.
If $\varepsilon$ is well spread, for example, when $\varepsilon_j = n^{-1/2}$ for all $1 \le j \le n$, then
\begin{equation}\label{application: Berry-Esseen bound}
|h_n(\varepsilon_1, ... , \varepsilon_n) - h_{\infty}(|\varepsilon|_2)| \le C \cdot |t|^3 \cdot \frac{E|X|^3}{\sqrt n }.
\end{equation}
Of course this bound does not hold for all $\varepsilon = (\varepsilon_1, ... , \varepsilon_n)$ on the unit  sphere $S^{n-1} = \{\varepsilon : |\varepsilon|_2 = 1 \}$. Consider a simple counter example. Let $X \sim \text{Uniform}([-\sqrt 3, \sqrt 3])$ and  $\varepsilon = e_1$.
Then $S_{\varepsilon} = X_1 \sim \text{Uniform}([-\sqrt 3, \sqrt 3])$, which is not Gaussian as $n \rightarrow \infty$.

Concerning expansions for weighted linear forms,  results of~\cite{KlarSodin2011} imply that the left hand side of~\eqref{application: Berry-Esseen bound} has order $\mathcal O(1/n)$ for a 'large' set of unit vectors $\varepsilon$.
The size of this set is measured according to the uniform probability measure, say $\sigma_{n-1}$, on the unit sphere $S^{n-1}$.

Let us now construct an asymptotic expansion using Theorem~\ref{asymptotic_expansion}. We have for any integer $s\geq 0$
$$
h_{\infty}(\lambda_1, ... , \lambda_s) = \E e^{i t (\lambda_1 X_1 + ... + \lambda_s X_s + G)}.
$$
Taking derivatives with respect to $\lambda_1, ... , \lambda_s$ at zero we get, for example, for $s\leq 2$
\begin{align*}
&\frac{\partial^3}{\partial \lambda_1^3} h_{\infty} (\lambda_1) \bigg |_{\lambda_1  = 0} = (it)^3 e^{-t^2/2} \beta_3, \\
&\frac{\partial^4}{\partial \lambda_1^4} h_{\infty} (\lambda_1) \bigg |_{\lambda_1  = 0} = (it)^4 e^{-t^2/2} \beta_4, \\
&\frac{\partial^4}{\partial \lambda_1^2 \partial \lambda_2^2} h_{\infty} (\lambda_1, \lambda_2) \bigg |_{\lambda_1  = 0, \lambda_2  = 0} = (it)^4 e^{-t^2/2} \beta_2^2,\\
&\frac{\partial^6}{\partial \lambda_1^3 \partial \lambda_2^3} h_{\infty} (\lambda_1, \lambda_2) \bigg |_{\lambda_1  = 0, \lambda_2  = 0} = (it)^6 e^{-t^2/2} \beta_3^2, \\
\end{align*}
where $\beta_2 = \E X^2 = 1, \beta_3 = \E X^3$ and $\beta_4 = \E X^4$. Substituting these equations to~\eqref{eq: expansion5} we get
\begin{align*}
h_n(\varepsilon_1, ... , \varepsilon_n) &= \E e^{i t G} + \frac{\varepsilon^3}{6} (it)^3 e^{-t^2/2} \beta_3   \\
&+\frac{\varepsilon^4}{24} [\beta_4 - 3](it)^4 e^{-t^2/2}  + \frac{( \beta_3 \varepsilon^3)^2}{72} (it)^6 e^{-t^2/2} + R_5.
\end{align*}
The expansion coincides with the well known Edgeworth expansion (involving cumulants) for characteristic functions of sums of random variables, see, e.g., \S~1, Ch.~6, in~\cite{Petrov1975}. It coincides as well with Edgeworth expansions for  expectations of smooth functions of sums of random vectors in Euclidean, resp. Banach spaces, see e.g., \cite{GotzeHipp1978} resp.  \cite{Gotze1989}.

Let us concentrate now on the so-called short asymptotic expansion
\begin{align}\label{application: short asymptotic expansion}
h_n(\varepsilon_1, ... , \varepsilon_n) = \E e^{i t G} + \frac{\varepsilon^3}{6} (it)^3 e^{-t^2/2} \beta_3 + R_4,
\end{align}
where
$$
|R_4|  \le C  \cdot |t|^4 \cdot \sum_{k=1}^n \varepsilon_k^4.
$$
It follows from~\cite[Lemma~4.1]{KlarSodin2011} that for some constants $C_1$ and $C_2$ and for all\\
$\rho : 1 > \rho > \exp (-C_1\,n)$ there exists a subset $B \subset S^{n-1}$ such that for any $\varepsilon \in B$ one has
\begin{align}\label{application: sphere prop}
\bigg |\sum_{k=1}^n \varepsilon_k^3 \bigg | \le \left(\log \frac 1 {\rho}\right)^2\, \frac{C_2}{n} \text{ and } \sum_{k=1}^n \varepsilon_k^4 \le \left(\log \frac 1 {\rho}\right)^2\, \frac{C_2}{n}
\end{align}
and $\sigma_{n-1}(B) \geq 1 - \rho$ for the uniform probability measure
$\sigma_{n-1}$  on the unit sphere $S^{n-1}$.

Thus combining~\eqref{application: short asymptotic expansion} and~\eqref{application: sphere prop}  we get
 for any $\rho : 1 > \rho > 0,$ and all $\varepsilon \in B$
\begin{equation}\label{KSbound}
|h_n(\varepsilon_1, ... , \varepsilon_n) - \E e^{i t G}| \leq C\,(|t|^3 + t^4)\,\left(\log \frac 1 {\rho}\right)^2\,
\frac {\beta_4}{n}
\end{equation}
for some constant $C$ (cf.~\eqref{application: Berry-Esseen bound}).

This property may be generalized to arbitrary functions $h_n(\varepsilon_1, ..., \varepsilon_n)$
which satisfies the conditions of Theorem~\ref{asymptotic_expansion}.

Extending this example it is possible to apply our result for asymptotic expansion in the central limit theorem for quadratic forms in sums of random elements with values in a Hilbert space including infinite dimensional cases,
see, e.g.~\cite{BhatRang2010},~\cite{Gotze2004},~\cite{GotzeZaitsev2014},~\cite{ProkhUlyanov2013},~\cite{ulyanov1986} and~\cite{ulyanov1987}.

Moreover, our result could be helpful in study of asymptotic expansions for the functionals of weighted sums of {\it dependent} random variables.

Let $X_1, \dots , X_n$ be identically distributed symmetric random variables and $\delta_1, \dots , \delta_n$ be independent Rademacher random variables, i.e. $\delta_i$ takes values $1$ and $-1$ with probabilities $1/2$. Assume that $\delta_1, \dots , \delta_n$ are independent of $X_1, \dots , X_n$. We emphasize that here it is not necessary that $X_1, \dots , X_n$ are independent. In order to in construct asymptotic expansions for   ${\E}F(\varepsilon_1\, X_1 + \dots + \varepsilon_n\, X_n)$ with some smooth measurable function $F$,  note that
$$
F(\varepsilon_1\, X_1 + \dots + \varepsilon_n\, X_n) \overset{d}{=} F(\varepsilon_1\, \delta_1\,  X_1 + \dots + \varepsilon_n\, \delta_n \, X_n),
$$
where $\overset{d}{=}$ denotes equality in distribution.  Consider functions
$$
h_n(\varepsilon_1,\ldots,\varepsilon_n)= {\E} F(\varepsilon_1\, \delta_1\,  X_1 + \dots + \varepsilon_n\, \delta_n \, X_n).
$$
The function $h_n$ satisfies the conditions~\eqref{cond_1}--\eqref{cond_3} provided $F$ is sufficiently smooth.

For instance, we can take for $i=1, \dots, n$ 
$$
X_i = \frac{Y_i}{\sqrt{\varepsilon_1^2\, Y_1^2 + \dots + \varepsilon_n^2\, Y_n^2}},
$$
where $Y_1, \dots , Y_n$ are independent random variables with common symmetric distribution. Then  $F(\varepsilon_1\, X_1 + \dots + \varepsilon_n\, X_n)$ is a function of a self-normalized weighted sum, see e.g. \cite{JingWang2010}.

On the other hand, in the special case $\varepsilon_1 = \dots = \varepsilon_n = 1/\sqrt{n}$ we may consider $F(X_1/\sqrt{n} + \dots +  X_n/\sqrt{n})$ as a function of exchangeable random variables, see e.g. \cite{GoetzeBloznelis2002}. 

\subsection{Expansion in the Free Central Limit theorem}
It has been shown in a recent paper~\cite{GotzeReshetenko2014} that one may apply the results of Theorem~\ref{asymptotic_expansion} in the setting of Free Probability theory.

Denote by $\mathcal M$ the family of all Borel probability measures defined on the real line $\R$. Let $X_1, X_2, \dots$ be free self-adjoint identically distributed random variables with distribution $\mu \in \mathcal M$. We always assume that $\mu$ has zero mean and unit variance.
Let $\mu_n$ be the distribution of the normalized sum $S_n : = \frac{1}{\sqrt n} \sum _{j=1}^n X_j$.
In free probability the sequence of measures $\mu_n$ converges to Wigner's semicircle law $\omega$. Moreover, $\mu_n$  is absolutely continuous with respect to the Lebesgue measure for sufficiently large $n$. We denote by $p_{\mu_n}$ the density of $\mu_n$. Define the Cauchy transform of a measure $\mu$:
$$G_\mu(z) = \int_{\R} \frac {\mu(dx)}{z - x},\ \ z\in \C^+,$$
where $\C^+$ denotes the upper half plane.

In \cite{ChG11v2} Chistyakov and G\"otze  obtained a formal power expansion for the Cauchy transform of $\mu_n$
and  an Edgeworth type expansions for $\mu_n$ and $p_{\mu_n}$. In~\cite{GotzeReshetenko2014} the general scheme from~\cite{Gotze1985} was applied to derive a similar result.

\subsection{Expansion of Quadratic von Mises Statistics}
Let $X, \overline X, X_1, ... , X_n$ be independent identically distributed random elements taking values in an arbitrary measurable space $(\mathcal X, \mathcal B)$.
Assume that $g: \mathcal X \rightarrow \R$ and $h: \mathcal X \times \mathcal X \rightarrow \R$  are real-valued measurable functions. In addition we assume that $h$ is symmetric.
We consider the quadratic functional
$$
w_n(\varepsilon_1, ... , \varepsilon_n) = \sum_{k=1}^n \varepsilon_j g(X_j) + \sum_{j,k=1}^n \varepsilon_j \varepsilon_k h(X_j, X_k),
$$
assuming that
$$
\E g(X) = 0, \quad \E(h(X, \overline X)|X) = 0.
$$
We shall derive an asymptotic expansion of $h_n(\varepsilon_1, ... , \varepsilon_n): = \E \exp(it w_n(\varepsilon_1, ... , \varepsilon_n))$.

Consider the measurable space $(\mathcal X, \mathcal B, \mu)$ with measure $\mu:= \mathcal L(X)$. Let $L^2:= L^2(\mathcal X, \mathcal B, \mu)$ denote the real Hilbert space of square integrable real functions.
A Hilbert-Schmidt operator $\mathbb Q: L^2 \rightarrow L^2$ is defined via
$$
\mathbb Q f(x) = \int_{\mathcal X} h(x,y) f(y) \mu(dy) = \E h(x, X) f(X), \quad f \in L^2.
$$
Let $\{e_j, j \geq 1\}$ denote a  complete orthonormal system of eigenfunctions of $\mathbb Q$ ordered by decreasing absolute values of the corresponding eigenvalues $q_1, q_2, ...  $, that is,
$|q_1| \geq |q_2| \geq ...$.
Then
$$
\E h^2(X, \overline X) = \sum_{j=1}^\infty q_j^2 < \infty, \quad h(x,y) = \sum_{j=1}^\infty q_j e_j(x) e_j(y)
$$
If the closed span $\langle \{e_j , j \geq 1\} \rangle \subset L^2$ is a proper subset, it might be necessary to choose functions $e_{-1}, e_0$ such that $\{ e_j, j = -1, 0, 1, ... \}$ is an orthonormal system and
\begin{align*}
g(x) = \sum_{k=0}^\infty g_k e_k(x), \quad h(x,x) = \sum_{k=-1}^\infty h_k e_k(x).
\end{align*}
It is easy to see that $\E e_j(X) = 0$ for all $j$. Therefore $\{e_j(X), j = -1, 0 , 1, ...\}$ is an orthonormal system of mean zero random variables.

We derive an expression for the derivatives of $h_\infty(\lambda_1, ... , \lambda_r)$.
Since for every fixed $k$ the sum $n^{-1/2} (e_k(X_1) + ... + e_k(X_n))$ weakly converges to a standard normal random variable we conclude
that $w_{n+r}(\lambda_1, ... , \lambda_r, n^{-1/2} , ... , n^{-1/2})$ weakly converges to the random
variable
\begin{align*}
w_\infty(\lambda_1, ... , \lambda_r)&:= w_r(\lambda_1, ... , \lambda_r) + \sum_{k=0}^\infty g_k Y_k + \sum_{k=1}^\infty q_k^2 (Y_k^2 - 1) \\
&+\E h(X, X) + 2\sum_{k=1}^\infty q_k \left (\sum_{l=1}^r \lambda_l e_k(X_l)\right) Y_k,
\end{align*}
where $Y_k , k \geq 0$ are independent standard normal random variables.
For every fixed $T$ we get by complex integration
$$
\E \exp \left [ it q_k (Y_k^2 - 1) + 2 it T Y_k \right ] = \frac{1}{\sqrt{1 - 2 i t q_k}} \exp(-i t q_k) \exp \left [  - \frac{2 t^2 T^2}{ \sqrt{1 - 2 i t q_k}} \right ].
$$
This yields
\begin{align} \label{applications: von mises derivatives}
h_\infty(\lambda_1, ... , \lambda_r) &= \varphi(t) \E \exp [ it w_r(\lambda_1, ... , \lambda_r) \\
&+(it)^2 \sum_{k=1}^\infty q_k T_k(\lambda)(2 q_k T_k(\lambda) + g_k)(1 -2it q_k)^{-1}], \nonumber
\end{align}
where $T_k(\lambda) = \sum_{l=1}^r \lambda_l e_k(X_l)$ and
\begin{align*}
\varphi(t) &= \left [ \prod_{k=1}^\infty \frac{1}{\sqrt{1-2it q_k}} \exp(-it q_k)\right ] \exp \left [ it \E h(X_1, X_1) - t^2 \sum_{k=0}^\infty g_k^2 (1 - 2 it q_k)^{-1} /2 \right ].
\end{align*}
Let us introduce the following functions of  $X$ and $\overline X$:
\begin{align*}
&h_t(X, \overline X): = h(X, \overline X) + 2 i t \sum_{k=1}^\infty q_k^2 e_k(X) e_k(\overline X)(1 - 2 it q_k)^{-1},\\
&g_t(X): = g(X) + i t \E\big( h_t(X, \overline X) g(\overline X) | X\big).
\end{align*}
Applying these notations we may rewrite~\eqref{applications: von mises derivatives} in the following way
$$
h_\infty(\lambda_1, ... , \lambda_r)  = \varphi(t) \E \exp \left [ it \sum_{j,k = 1}^n h_t(X_j, X_k) \lambda_j \lambda_k + it \sum_{j=1}^r \lambda_j g_t(X_j) \right ].
$$
Taking  derivatives of $h_\infty$ with respect with $\lambda_1, ... , \lambda_r$ at zero we get
$$
h_n(\varepsilon_1, ... , \varepsilon_n) = \varphi(t) \sum_{r = 0}^{s-3} a_r(t, h , g) + R_s,
$$
where
\begin{align*}
&a_r(t,h,g): = P_r(\ve^{*} \kappa_{*}) \E \exp \left [ it \sum_{j,k = 1}^n h_t(X_j, X_k) \lambda_j \lambda_k + it \sum_{j=1}^r \lambda_j g_t(X_j) \right ]\bigg |_{\lambda_1 = ... = \lambda_r = 0}.
\end{align*}
Higher order $U$-statistics may be treated by similar arguments. See, for example, the result of~\cite{Gotze1989} and~\cite{BentkusGotze1999}.

\subsection{Expansions for Weighted One Sided Kolmogorov-Smirnov Statistics}
Let $X_1, ... X_n$ be an independent identically distributed random variables with uniform distribution in $[0, 1]$. Consider the following statistic
$D^{+}(\varepsilon_1, ... , \varepsilon_n, t) = \sum_{j=1}^n \varepsilon_j (\mathbb I(X_j \le t) - t)$. For example, if  $\varepsilon_j = n^{-1/2}, j = 1, ... , n$ then we have
$D^{+}(t) = n^{1/2}(F_n(t) -t)$, where $F_n(t)$ denotes the empirical distribution function of $X_1, ... , X_n$.
We are interested in the asymptotic expansion of
$$
\Pb(\sup_{0 \le t \le 1} D^{+}(\varepsilon_1, ... , \varepsilon_n, t) > a), \quad a > 0.
$$
It is well known that $h_\infty(0) = \exp[-2 a^2]$ and
\begin{align*}
h_\infty(\lambda) &= \int_0^1 \Pb(x(t) + \lambda(\mathbb I(s < t) - t) > a, 0 \le t \le 1) \, ds \\
&=\int_0^1 \E f_a(s, x(s), \lambda) f_a(1-s, x(s), -\lambda) \, ds,
\end{align*}
where $f_a(s,x,\lambda) = \Pb(x(t) > a + \lambda t, 0\le t \le s | x(s) = x) = \exp(-2a(a + \lambda s - x)/s)$ and $x(t), 0 \le t \le 1$ is a Brownian bridge. For more details, see~\cite{Gotze1985}.
Then it follows from Theorem~\ref{asymptotic_expansion}
that
$$
\Pb(\sup_{0 \le t \le 1} D^{+}(\varepsilon_1, ... , \varepsilon_n, t) > a) = \left [ 1 + \frac{1}{6} \varepsilon^3 \frac{\partial }{\partial a} + \mathcal O(|\varepsilon|_4^4) \right ] \exp(-2 a^2).
$$
Such expansions for equal weights have been derived for example by combinatorial and analytic techniques in~\cite{Lauwerier1963} and \cite{Gnedenko1961}.

\section{Acknowledgements}

We would like to thank the Associate Editor and the  Reviewer for helpful
comments and suggestions.

\def\polhk#1{\setbox0=\hbox{#1}{\ooalign{\hidewidth
  \lower1.5ex\hbox{`}\hidewidth\crcr\unhbox0}}}

\end{document}